\documentclass{amsart}
\usepackage{amssymb}
\usepackage{pdflscape}
\usepackage{amscd}
\usepackage{fancyhdr}

\newtheorem{theorem}{Theorem}[section]
\newtheorem{corollary}[theorem]{Corollary}
\newtheorem{lemma}[theorem]{Lemma}
\newtheorem{proposition}[theorem]{Proposition}

\newtheorem{example}[theorem]{Example}
\newtheorem{definition}[theorem]{Definition}

\title{Counting generating spaces of matrices}
\author{Markus Reineke}

\begin{document}
\begin{abstract} We prove that generating subspaces of matrix rings over finite fields are counted by polynomials. We use this result to define and study two-variable versions of polynomials counting isomorphism classes of absolutely irreducible representations of free algebras.  \end{abstract}
\maketitle
\parindent0pt

\section{Introduction}
It is a common theme in many areas of algebra, notably linear algebra, group theory and representation theory, to count interesting classes of objects over finite fields, often (but not always) exhibiting polynomiality in the size of the field.\\[1ex]
In the present work, we solve this counting problem for spaces of square matrices which generate the matrix algebra (Theorem \ref{t1}), by determining counting polynomials $s_d^{(m)}(q)$, thus precisely quantifying the slogan that ``most'' at least two-dimensional subspaces of matrices are generating.\\[1ex]
We reduce the counting problem to counting isomorphism classes of absolutely irreducible representations of free algebras, for which counting polynomials $a_d^{(m)}(q)$ were determined in \cite{MR,R} (the proofs there being highly indirect).
\\[1ex]Despite being interesting in their own right, our counting polynomials are shown (Corollary \ref{mahler}) to appear naturally as ``Mahler-type'' coefficients in a two-variable version $a_d(q,u)$ of the counts of absolutely irreducible representations, which elucidate the dependence of the $a_d^{(m)}(q)$ on the rank $m$ of the free algebra, in the spirit of a similar result for Kac polynomials of quivers in \cite{HeR}.\\[1ex]
The proofs of these results are based on direct linear algebra constructions, natural scheme structures on the relevant classes of objects, the lambda-ring techniques of \cite{MR}, and identities between Gaussian binomial coefficients; all relevant techniques are recalled in Section \ref{pre}. The counting polynomials $s_d^{(m)}(q)$ are constructed in Section \ref{sdmq}; there we also list various small-dimensional cases and describe to which extent they are approximated by Gaussian binomial coefficients. The two-parameter counting polynomials $a_d(q,u)$ are constructed in Section \ref{adqu}. Again, we list small-dimensional examples, and we exhibit a natural factorization (Theorem \ref{fact}), and describe highest- and lowest-degree terms in the new variable.\\[2ex]
{\bf Acknowledgments:} The author would like to thank S.~Mozgovoy for several discussions related to the topics of this work.

\section{Prerequisites}\label{pre}

Let $k$ be a finite field, and let $\overline{k}$ be an algebraic closure of $k$. Denote by $M_d(k)$ the $k$-algebra of $d\times d$-matrices with entries in $k$. \\[1ex]
An $m$-tuple $$A_*=(A_1,\ldots,A_m)$$ in $M_d(k)$ defines a representation $(V,\rho)$ of the free algebra $$F^m(k)=k\langle x_1,\ldots,x_m\rangle$$ on $V=k^d$ and vice versa. Two such representations are isomorphic if and only if the corresponding matrix tuples $A_*$, $B_*$ are simultaneously conjugate, that is, $$B_k=gA_kg^{-1}$$ for $k=1,\ldots,m$, for some $g\in{\rm GL}_d(k)$.\\[1ex]
The representation $V$ is irreducible if there are no proper non-zero subspaces $U\subset V$ which are $F^m(k)$-stable, that is, such that $$A_k(U)\subset U$$ for $k=1,\ldots,m$. It is called absolutely irreducible if its base extension $\overline{V}=\overline{k}\otimes_kV$ is an irreducible representation of $F^m(\overline{k})$.

\begin{theorem} There exist polynomials $a_d^{(m)}(q)\in\mathbb{Z}[q]$ such that $a_d^{(m)}(|k|)$ equals the number of isomorphism classes of $d$-dimensional absolutely irreducible representations of $F^m(k)$ for all finite fields $k$.
\end{theorem}

We give an explicit formula for these polynomials in terms of plethystic exponentials, following closely the notation of \cite{MR}. We consider the complete $\mathbb{Q}(q)$-algebra $R=\mathbb{Q}(q)[[t]]$ with its $t$-adic topology and denote its maximal ideal by $\mathfrak{m}=t\mathbb{Q}(q)[[t]]$. Define ${\rm Exp}:\mathfrak{m}\rightarrow1+\mathfrak{m}$ as the unique continuous homomorphism from the additive group of $\mathfrak{m}$ to the multiplicative group of $1+\mathfrak{m}$ such that $${\rm Exp}(q^it^d)=(1-q^it^d)^{-1}.$$ It is in fact an isomorphism, whose inverse is denoted by ${\rm Log}$. For fixed $m\geq 0$, we define a $\mathbb{Q}(q)$-linear continuous twist operator $T$ on $R$ by $$T(t^d)=(-1)^dq^{(1-m)d(d-1)/2}t^d,$$
and we consider the following series of $q$-hypergeometric type:
$$F(t)=\sum_{d\geq 0}\frac{q^{md(d+1)/2}t^d}{(1-q)\cdot\ldots\cdot(1-q^d)}\in R.$$

\begin{theorem}\label{mr} In the ring $R$, we have an identity
$$\sum_{d\geq 1}a_d^{(m)}(q)t^d=(1-q)\cdot{\rm Log}(T^{-1}F(t)^{-1}).$$
\end{theorem}

\begin{proof} Defining a twisted $\mathbb{Q}(q)$-linear continuous multiplication on $R$ by
$$t^d\circ t^e=q^{(m-1)de}t^{d+e},$$
we have by \cite[Theorem 5.1]{MR}:
$$\sum_{d\geq 0}\frac{q^{(m-1)d^2}t^d}{(1-q^{-1})\cdot\ldots\cdot(1-q^{-m})}\circ{\rm Exp}(\frac{1}{1-q}\sum_{d\geq 1}a_d^{(m)}(q)t^d)=1.$$
Noting that $$T(t^d\circ t^e)=T(t^d)\cdot T(t^e),$$ applying $T$ to the previous equation, and resolving for the generating series of the $a_d^{(m)}(q)$, we arrive at the claimed formula.\end{proof}

Following \cite[Section 6]{HR}, if $X$ is a separated scheme of finite type over a finite field $k$, we say that $X$ has polynomial count if there exists a polynomial $P_X(q)\in\mathbb{C}[q]$ such that $P_X(|k'|)$ equals the number $|X(k')|$ of $k'$-rational points of $k$, for all finite extension fields $k\subset k'$. 

\begin{lemma}\label{pc} If $X$ has polynomial count, its counting polynomial $P_X(q)$ already belongs to $\mathbb{Z}[q]$, and is of degree $\dim X$.
\end{lemma}

Later on, we will make frequent use of Gaussian binomial coefficients and identities related to them. For $a\in\mathbb{Z}$ and $b\in\mathbb{N}$, we define
$$\left[{a\atop b}\right]_q=\frac{(q^a-1)\cdot\ldots\cdot(q^{a-b+1}-1)}{(q^b-1)\cdot\ldots\cdot (q-1)},$$
which, for $a\geq 0$, is the counting polynomial for the Grassmannian ${\rm Gr}_b(W)$ of $b$-dimensional subspaces of an $a$-dimensional $k$-vector space $W$. In particular, we write $$[a]_q=\left[{a\atop 1}\right]_q.$$ The following identities hold for $a,a',a''\in\mathbb{N}$:
\begin{equation}\label{i1}\left[{a'+a'\atop b}\right]_q=\sum_{b=b'+b''}q^{(a'-b')b''}\left[{a'\atop b'}\right]_q\left[{a''\atop  b''}\right]_q,\end{equation}
\begin{equation}\label{i2}\left[{a\atop b}\right]_q=(-1)^bq^{ab-b(b-1)/2}\left[{-a+b-1\atop b}\right]_q,\end{equation}
\begin{equation}\label{i3}\sum_{b=0}^a(-1)^bq^{b(b-1)/2}\left[{a\atop b}\right]_q=\delta_{a,0},\end{equation}
\begin{equation}\label{i4}(u-1)\cdot\ldots\cdot(u-q^{a-1})=\sum_{b=0}^a(-1)^bq^{b(b-1)/2}\left[{a\atop b}\right]_qu^{a-b}.\end{equation}

Finally, we will need the following formula for the map ${\rm Log}:1+\mathfrak{m}\rightarrow\mathfrak{m}$:
$${\rm Log}=\Psi^{-1}\circ\log,$$
where $$\Psi^{-1}(P(q,t))=\sum_{i\geq 1}\frac{\mu(i)}{i}P(q^i,t^i)$$
for the number-theoretic Moebius function $\mu$.

\section{Polynomials counting generating subspaces}\label{sdmq}

We first translate the property of generation of a matrix algebra into representation theory.

\begin{lemma}\label{key} The tuple $A_*$ generates $M_d(k)$ as a (unital) $k$-algebra if and only if the representation $V$ is absolutely irreducible.
\end{lemma}

\begin{proof} If the tuple $A_*$ generates $M_d(k)$ as $k$-algebra, it also generates $M_d(\overline{k})$ as $\overline{k}$-algebra. In particular, a subspace $U\subset \overline{k}$ fixed by the $A_k$ is fixed by every matrix, thus trivial. Thus the representation defined by $A_*$ is absolutely irreducible. Conversely, assume that the $d$-dimensional representation $V$ is absolutely irreducible, thus $\overline{V}$ is irreducible. Its endomorphism ring thus reduces to the scalars, and by Jacobson density, the map $F^m(\overline{k})\rightarrow{\rm End}_{\overline{k}}(\overline{V})$ is surjective. Consequently, also $$F^m(k)\rightarrow{\rm End}(V) \simeq M_d(k)$$
is surjective, proving the generation property.\end{proof}

 \begin{definition} For $d\geq 1$ and $m\leq d^2$ we define $S_d^{(m)}(k)$ as the set of $m$-dimensional subspaces of $M_d(k)$ which generate $M_d(k)$ as a unital $k$-algebra.
 \end{definition}
 
 \begin{theorem}\label{t1} There exist polynomials $s_d^{(m)}(q)\in\mathbb{Z}[q]$ such that $s_d^{(m)}(|k|)$ equals the cardinality of $S_d^{(m)}(k)$, for all finite fields $k$.
 \end{theorem}
 
\begin{proof} For an $n$-dimensional $k$-vector space $W$ and positive integers $m$ and $r$, define $W^m(r)\subset W^m$ as the subset of tuples $(w_1,\ldots,w_m)$ in $W$ which span an $r$-dimensional subspace. Then $W^m$ is the disjoint union of the $W^m(r)$ for $0\leq r\leq m$. We have a map $$p:W^m(r)\rightarrow {\rm Gr}_{r}(W)$$ to the set of $r$-dimensional subspaces of $W$, mapping a tuple to the subspace generated by it. This map is surjective and equivariant with respect to the natural ${\rm GL}(W)$-action, with fibre over a fixed $U\subset W$ being isomorphic to $U^{m}(r)$, the set of generating $m$-tuples in the $r$-dimensional space $U$. Choosing a basis of $U$, this can be identified with the set of $r\times m$-matrices over $k$ of rank $r$. The number of such matrices, for $q=|k|$, equals
 $$(q^m-1)\cdot\ldots\cdot(q^m-q^{r-1}).$$
 We apply this to the $k$-vector space $W=M_d(k)$, thus we have a disjoint union
 $$M_d(k)^m=\bigcup_{r=0}^{d^2}M_d(k)^m(r).$$
 Denoting by $M_d(k)_{\rm ai}^m\subset M_d(k)^m$ the set of $m$-tuples defining an absolutely irreducible tuple (equivalently, by Lemma \ref{key}, generating the matrix ring), this restricts to a decomposition
 $$M_d(k)_{\rm ai}^m=\bigcup_{r=0}^{d^2}M_d(k)_{\rm ai}^m(r),$$
 where of course
 $$M_d(k)_{\rm ai}^m(r)=M_d(k)_{\rm ai}^m\cap M_d(k)^m(r).$$
 Under the map $p: M_d(k)^m(r)\rightarrow {\rm Gr}_r(M_d(k))$, the subset $M_d(k)^m_{\rm ai}(r)$ maps precisely to $S_d^{(r)}(k)$ as long as $r\leq m$. We thus have:
 \begin{equation}\label{eqn0}|M_d(k)_{\rm ai}^m|=\sum_{r=0}^{d^2}(q^m-1)\cdot\ldots\cdot(q^m-q^{r-1})\cdot|S_d^{(r)}(k)|.\end{equation}
The automorphism group of an absolutely irreducible representation reducing to the scalars $k^*$, its simultaneous conjugacy class under the ${\rm GL}_d(k)$-action has cardinality
$$|{\rm PGL}_d(k)|=(q^d-1)\cdot\ldots\cdot(q^d-q^{d-1})/(q-1).$$
Thus we have $$|M_d(k)_{\rm ai}^m|=|{\rm PGL}_d(k)|\cdot a_d(|k|),$$
and in particular, $|M_d(k)_{\rm ai}^m|$ behaves polynomially in $k$.  The equations (\ref{eqn}) for $m=0,\ldots,d^2$, form  a system of linear equations determining the $|S_d^{(m)}(k)|$ from the $|M_d(k)_{\rm ai}^m|$, with the coefficient matrix being triangular since the coefficient vanishes for $r>m$. The diagonal coefficients equal $$(q^m-1)\cdot\ldots\cdot(q^m-q^{m-1})\not=0,$$ thus the  cardinalities $|S_d^{(m)}(k)|$ can be expressed as $\mathbb{Q}(q)$-linear combinations of the cardinalities $|M_d(k)_{\rm ai}^m|$. This proves that the claimed $s_d^{(m)}(q)$ exists as an element of $\mathbb{Q}(q)$. Since it assumes only integer values, we have $s_d^{(m)}(q)\in\mathbb{Q}[q]$.\\[1ex]
To prove that it is already an element of $\mathbb{Z}[q]$, we argue geometrically and view $M_d(k)^m$ and all related sets as schemes over $k$, compare \cite[Section 6]{R}. Then $M_d(k)^m_{\rm ai}$ is a Zariski-open subset, ${\rm Gr}_r(M_d(k))$ is a projective scheme over $k$, and again $S_d^{(r)}(k)$ is a Zariski-open subset. By Lemma \ref{pc}, we thus find $s_d^{(m)}(q)\in\mathbb{Z}[q]$.\end{proof}

For further computations, we note the linear equation relating the two classes of counting polynomials as a separate result:

\begin{corollary}\label{eqn} For all $m\geq 0$, we have 
$$a_d^{(m)}(q)=\frac{q-1}{(q^d-1)\cdot\ldots\cdot(q^d-q^{d-1})}\cdot\sum_{r=0}^{d^2}(q^m-1)\cdot\ldots\cdot(q^m-q^{r-1})s_d^{(r)}(q).$$
\end{corollary}

It is not difficult to resolve this recursion.

\begin{lemma}\label{inv} We have
$$s_d^{(m)}(q)=\sum_{r=1}^m\frac{(-1)^{m-r}q^{r(r+1)/2-mr+d(d-1)/2}(q^d-1)\ldots\ldots\cdot(q^2-1)}{(q^r-1)\cdot\ldots\cdot(q-1)\cdot(q^{m-r}-1)\cdot\ldots\cdot(q-1)}a_d^{(r)}(q).$$
\end{lemma}

\begin{proof} Using the formula of Corollary \ref{eqn}, the right hand side of the claimed equality can be rewritten as
$$\sum_{s\leq r}\frac{(-1)^{m-r}q^{r(r+1)/2-mr+s(s-1)/2}}{(q^{m-r}-1)\cdot\ldots\cdot(q-1)\cdot(q^{r-s}-1)\cdot\ldots\cdot(q-1)}
s_d^{(s)}(q)=$$
$$=\sum_sq^{s(s-1)/2}\sum_{t=0}^{m-s}\frac{(-1)^tq^{t(t-1)/2-m(m-1)/2}}{(q^t-1)\cdot\ldots\cdot(q-1)\cdot(q^{m-s-t}-1)\cdot\ldots\cdot(q-1)}s_d^{(s)}(q)$$
with the substitution $t=m-r$. Applying identity (\ref{i3}), we see that the inner sum vanishes unless $s=m$, and the term simplifies to $s_d^{(m)}(q)$, as claimed.\end{proof}
 
 \begin{example}  We have $s_1^{(1)}(q)=1$ trivially, and $s_d^{(1)}(q)=0$ for all $d\geq 2$, since a single matrix admits an eigenvector over $\overline{k}$, and thus a proper subrepresentation (equivalently, since $a_d^{(1)}(q)=0$). The previous recursion has a simple expression for $m=2,3$:
 $$s_d^{(2)}(q)=\frac{|{\rm PGL}_d(k)|}{(q^2-1)(q^2-q)}\cdot a_d^{(2)}(q).$$
 
 $$s_d^{(3)}(q)=\frac{|{\rm PGL}_d(k)|}{(q^3-1)(q^3-q)(q^3-q^2)}\cdot( a_d^{(3)}(q)-(q^2+q+1)a_d^{(2)}(q)).$$
 \end{example}
 
\begin{example}  We have
$$s_2^{(2)}=q^4,\; s_2^{(3)}=q^3+q^2,\, s_2^{(4)}=1.$$
For $d=3$, we find

\begin{eqnarray*}
s_3^{(2)}&=&q^{14}+q^{13}-q^{11}-q^{10},\\
s_3^{(3)}&=&q^{18}+q^{17}+2q^{16}+3q^{15}+2q^{14}+q^{13}-2q^{11}-3q^{10}-2q^9-q^8,\\
s_3^{(4)}&=&q^{20} + q^{19} + 2q^{18} + 3q^{17} + 5q^{16} + 6
q^{15} + 6q^{14} + 5q^{13} + 3q^{12} - 3q^{10} -\\&&- 5q^{
9} - 5q^{8} - 3q^{7}-q^6,\\
s_3^{(5)}&=&q^{20} + q^{19} + 2q^{18} + 3q^{17} + 5q^{16} + 6
q^{15} + 8q^{14} + 9q^{13} + 9q^{12} + 7q^{11} +\\ &&+ 4q^{
10} + q^{9} - 2q^{8}- 4q^{7} - 5q^{6} - 4q^{5} - 2q^{4},\\
  s_3^{(6)}&=&q^{18} + q^{17} + 2q^{16} + 3q^{15} + 4q^{14} + 5
q^{13} + 7q^{12} + 7q^{11} + 8q^{10} + 8q^{9} +\\ &&+ 6q^{8
} + 3q^{7} + q^{6} - q^{5}
  - 2q^{4} - 2q^{3} - 2q^{2} - q,\\
  s_3^{(7)}&=&q^{14} + q^{13} + 2q^{12} + 2q^{11} + 3q^{10} + 3q^{9} + 
4q^{8} + 4q^{7} + 4q^{6} + 3q^{5} + 3q^{4}+\\ && + 2q^{3}-q-1,\\
s_3^{(8)}&=&q^{8} + q^{7} + q^{6} + q^{5} + q^{4} + q^{3} + q^{2} + q + 1,\\
 s_3^{(9)}&=&1.\end{eqnarray*}
\end{example}

 Since the property of a subspace to be generating is open in the Zariski topology, it is reasonable to define
 $$r_d^{(m)}(q)=\left[{d^2\atop m}\right]_q-s_d^{(m)}(q).$$
 
 Namely, since $\left[{d^2\atop m}\right]_q$ is the counting polynomial for ${\rm Gr}_m(M_d(k))$ and $s_d^{(m)}(q)$ is the counting polynomial for the open subset $S_d^{(m)}(k)\subset{\rm Gr}_m(M_d(k))$, we see that $r_d^{(m)}(q)$ is the counting polynomial for the complement
 $$R_d^{(m)}(k):={\rm Gr}_m(M_d(k))\setminus S_d^{(m)}(k),$$
 the closed subscheme of spaces generating a proper subalgebra of $M_d(k)$.
 
 \begin{proposition}\label{propdeg} We have $\deg r_d^{(m)}(q)\leq m(d^2-m)-(m-1)(d-1)$.
 \end{proposition}
 
 \begin{proof}  Suppose $U\subset M_d(k)$ belongs to $R_d^{(m)}(k)$, thus generates a proper subalgebra $A$ of $M_d(k)$. Then $\overline{U}=\overline{k}\otimes_kU$ generates the proper subalgebra $$\overline{A}=\overline{k}\otimes_kA\subset\overline{k}\otimes_kM_d(k)\simeq M_d(\overline{k}).$$ Thus the representation of $\overline{A}$ on $\overline{k}^d$ cannot be irreducible by Lemma \ref{key}, and therefore admits a proper non-zero subrepresentation. Thus $\overline{A}$ is contained in some conjugate of an algebra $\mathfrak{p}_{e,f}$ of block-upper triangular matrices of block sizes $e$ and $f$, where $d=e+f$ and $e,f\geq 1$. In other words, $\overline{A}$ belongs to
 $$X_{e,f}=\{g\mathfrak{p}_{e,f}g^{-1}\, :\, g\in{\rm GL}_d(\overline{k})\}\subset{\rm Gr}_{d^2-ef}(M_d(\overline{k})).$$ %where $\mathfrak{p}_{e,f}$ is  the algebra of block-upper triangular matrices of block sizes $e$ and $f$.
 The set $X_{e,f}$ is the closed image of the map $${\rm GL}_d(\overline{k})/P_{d,e}\rightarrow {\rm Gr}_{d^2-ef}(M_d(\overline{k}))$$ given by conjugating $\mathfrak{p}_{e,f}$, where $$P_{e,f}={\rm GL}_d(\overline{k})\cap \mathfrak{p}_{e,f}$$
 is a parabolic subgroup.  Thus $X_{e,f}$ is a closed subscheme of dimension at most $ef$. We consider the scheme
 $${\rm Fl}_{m,d^2-ef}(M_d(\overline{k}))=\{(U,V)\, :\, U\subset V\}\subset {\rm Gr}_m(M_d(\overline{k}))\times {\rm Gr}_{d^2-ef}(M_d(\overline{k})).$$
The projection $p_2$ to the second component is equivariant for the natural action of ${\rm GL}(M_d(\overline{k}))$, with fibre over a fixed $V$ isomorphic to ${\rm Gr}_m(V)$, which has dimension $m(d^2-ef-m)$. We also consider the projection $p_1$ to the first component. By definition, $R_d^{(m)}(\overline{k})$ is contained in  the union of the $$p_1p_2^{-1}X_{e,f},$$
since $$\overline{U}\subset\overline{A}\subset g\mathfrak{p}_{d,e}g^{-1}$$
as above. We can thus estimate, for all $e,f\geq 1$ such that $e+f=d$:
$$\dim p_1p_2^{-1}X_{e,f}\leq\dim p_2^{-1}X_{e,f}\leq m(d^2-ef-m)m+ef=$$
$$=m(d^2-m)-(m-1)ef\leq m(d^2-m)-(m-1)(d-1).$$
Since the degree of the counting polynomial $r_d^{(m)}(q)$ equals the dimension of $R_d^{(m)}(\overline{k})$, this proves the claim.\end{proof}

\begin{corollary} For all $r=0,\ldots,d-2$, we have $$s_d^{(d^2-r)}(q)=\left[{d^2\atop r}\right]_q=\frac{(q^{d^2}-1)\cdot\ldots\cdot(q^{d^2-r+1}-1)}{(q^r-1)\cdot\ldots\cdot(q-1)}.$$
\end{corollary}

\begin{example} We have $$r_4^{(9)}(q)=1+q+q^2-q^4-2q^5+q^6+\ldots+2q^{39},$$ so in particular, $r_d^{(m)}(q)$ can have negative coefficients (although all $r_d^{(m)}(q)$ for $d\leq 3$ have nonnegative coefficients).\end{example}
 
 To finish this section, we observe a vanishing property for the $s_d^{(m)}(q)$:
 
 \begin{lemma} We have  $s_d^{(m)}(1)=0$ if $m\leq d-1$.
 \end{lemma}
 
 \begin{proof} The coefficients in the summation in Lemma \ref{inv} have a zero at $q=1$ of order $d-1$ in the numerator, and of order $m$ in the denominator. Thus the coefficients vanish at $q=1$ for $m\leq d-1$.\end{proof}

\section{Two-parameter versions of counting polynomials for absolutely irreducible representations}\label{adqu}
 
For the following, recall from Corollary \ref{eqn} that  $$a_d^{(m)}(q)=\frac{q-1}{(q^d-1)\cdot\ldots\cdot(q^d-q^{d-1})}\cdot\sum_{r=1}^{d^2}(q^m-1)\cdot\ldots\cdot(q^m-q^{r-1})s_d^{(r)}(q)$$
 for all $m\geq 0$.\\[1ex]
We define two-variable versions of the polynomials $a_d^{(m)}(q)$ formally, by generalizing the explicit formula Theorem \ref{mr} for the $a_d^{(m)}(q)$. We consider the complete $\mathbb{Q}(q,u)$-algebra $\widehat{R}=\mathbb{Q}(q,u)[[t]]$ with its adic topology and denote its maximal ideal by $\mathfrak{m}$. Again we define ${\rm Exp}:\mathfrak{m}\rightarrow1+\mathfrak{m}$ as the unique continuous homomorphism from the additive group of $\mathfrak{m}$ to the multiplicative group of $1+\mathfrak{m}$ such that $${\rm Exp}(q^iu^nt^d)=(1-q^iu^nt^d)^{-1},$$ and we denote its inverse by ${\rm Log}$. We define a $\mathbb{Q}(q,u)$-linear continuous twist operator $T$ on $\widehat{R}$ by $$T(t^d)=(-1)^d(q^{-1}u)^{-d(d-1)/2}t^d.$$
 
\begin{definition}\label{defadqu} Define polynomials $a_d(q,u)\in\mathbb{Q}(q)[u]$ by
$$\sum_{d\geq 1}a_d(q,u)t^d=(1-q)\cdot{\rm Log}(T^{-1}(\sum_{d\geq 0}\frac{u^{d(d+1)/2}t^d}{(1-q)\cdot\ldots\cdot(1-q^d)})^{-1}).$$
\end{definition}

 That $a_d(q,u)$ is indeed a polynomial in $u$ can be seen from the fact that the sum on the right hand side has coefficients in $\mathbb{Q}(q)[u]$, and both inverting the series and taking ${\rm Log}$ of it preserve polynomiality in $u$.\\[1ex]
It follows immediately from this definition, together with Theorem \ref{mr}, that
\begin{lemma} For all $m\geq 0$, we have $$a_d(q,q^m)=a_d^{(m)}(q).$$
\end{lemma}

Thus the identity at the beginning of this section now implies

\begin{corollary} We have

  $$a_d(q,u)=\frac{q-1}{(q^d-1)\cdot\ldots\cdot(q^d-q^{d-1})}\cdot\sum_{r=1}^{d^2}(u-1)\cdot\ldots\cdot(u-q^{r-1})s_d^{(r)}(q).$$
  \end{corollary}
 %Thus
 
 %$$a_d(q,u)=\frac{q-1}{(q^d-1)\cdot\ldots\cdot(q^d-q^{d-1})}\cdot\sum_{r=0}^{d-2}(u-1)\cdot\ldots\cdot(u-q^{d^2-r-1})\left[{d^2\atop r}\right]+$$
%$+$ terms of $u$-degree $\leq d^2-d+1$.
 
 We recall the Mahler-type expansion of \cite[Theorem 6.1]{HeR}.
 
 \begin{theorem} If $f(q,u)\in\mathbb{Q}(q)[u]$ is a polynomial such that $f(q,q^m)\in\mathbb{Z}[q]$ for all $m\geq 0$, then $f(q,u)$ admits an expansion of the form $$f(q,u)=\sum_lc_l(q)\left\langle{u\atop l}\right\rangle_q,$$ where
 $$\left\langle{u\atop l}\right\rangle_q=\prod_{i=1}^l\frac{q^{1-i}u-1}{q^i-1}$$
 and $c_l(q)\in\mathbb{Z}[q]$, non-zero only for finitely many $l$.
 \end{theorem}
 
 Thus we see that the counting polynomials $s_d^{(m)}(q)$ naturally appear in the Mahler-type expansion of the $a_d(q,u)$:
 
 \begin{corollary}\label{mahler} The Mahler-type expansion of $a_d(q,u)$ takes the form
   $$a_d(q,u)=\frac{q-1}{(q^d-1)\cdot\ldots\cdot(q^d-q^{d-1})}\cdot\sum_{r=1}^{d^2}s_d^{(r)}(q)q^{r(r-1)/2}(q^r-1)\cdot\ldots\cdot(q-1)\cdot\left\langle{u\atop r}\right\rangle.$$
   \end{corollary}
 
\begin{example} We make the $a_d(q,u)$, in factorized form, explicit for small $d$. We have $a_1(q,u)=u$ and
$$a_2(q,u)=\frac{u^2(u-1)(u-q)}{q(q-1)(q+1)}.$$
We have
$$a_3(q,u)=\frac{u^3(u-1)(u-q)(u+q)(u^3+u^2-(q+1)^2u+q^2(q+1))}{q^3(q-1)^2(q+1)(q^2+q+1)}.$$
We have $a_4(q,u)=$
\begin{eqnarray*}&=&\frac{u^4(u-1)(u-q)}{q^6(q-1)^3(q+1)^2(q^2+1)(q^2+q+1)}\cdot\left(u^{10}+(q+1)u^9+(q^2+q+1)u^8-\right.\\
&&-(q+1)(q^2+1)u^7-(2q^4+4q^3+5q^2+4q+2)u^6+\\
&&+(q+1)(q^2-q+1)(q^2+q+1)u^5+q(q^2+q+1)(q^3+q^2+1)u^4+\\
&&+q^2(q+1)(q^2+1)(q^2+q+1)u^3-q^4(q+1)(q^2+1)(q^2+q+1)u^2-\\
&&\left.-q^4(q+1)(q^2+1)(q^2+q+1)u+q^6(q^2+1)(q^2+q+1)\right).\end{eqnarray*}

\end{example}
 
 %$$\sum_{d\geq 0}\frac{u^{d^2}t^d}{(q^d-1)\cdot\ldots\cdot(q^d-q^{d-1})}\circ{\rm Exp}(\frac{1}{1-q}\cdot\sum_{d\geq 1}a_d(q,u)t^d)=1,$$
 %$t^d\circ t^e=(q^{-1}u)^{de}\cdot t^{d+e}$, and ${\rm Exp}$ is defined with respect to $\psi_n:q\mapsto q^n,\, t\mapsto t^n,\, u\mapsto u^n$. With $T(t^d)=(q^{-1}u)^{-d(d-1)/2}(-1)^dt^d$, we find
 %$$(\sum_{d\geq 0}\frac{u^{d(d+1)/2}t^d}{(1-q)\cdot\ldots\cdot(1-q^d)})^{-1}=T({\rm Exp}(\frac{1}{1-q}\cdot\sum_{d\geq 1}a_d^(q,u)t^d)).$$
 
We can now prove several factorization properties of the $a_d(q,u)$ and describe the highest $u$-terms:
 
\begin{theorem}\label{fact} We have $a_1(q,u)=u$ and, for all $d\geq 2$, $$a_d(q,u)=\frac{(q-1)u^d(u-1)(u-q)}{(q^d-1)\cdot\ldots\cdot(q^d-q^{d-1})}\cdot\overline{a}_d(q,u),$$
where
$$\overline{a}_d(q,u)=\sum_{r=0}^{d-2}[r+1]_qu^{d^2-d-2-r}+\mbox {terms of lower $u$-degree}.$$

\end{theorem}

\begin{proof} Divisibility of $a_d(q,u)$ by $u-1$ and $u-q$ for $d\geq 2$ follows from $$a_d^{(0)}(q)=0=a_d^{(1)}(q).$$  To see divisibility by $u^d$, we rewrite the definition of $a_d(q,u)$ as
\begin{equation*}\label{rewrite}\sum_{d\geq 1}(u^{-d}a_d(q,u))(ut)^d=(1-q)\cdot{\rm Log}(T^{-1}(\sum_{d\geq 0}\frac{u^{d(d-1)/2}(ut)^d}{(1-q)\cdot\ldots\cdot(1-q^d)})^{-1}).\end{equation*}
The polynomiality argument given after Definition \ref{defadqu} proves that $$u^{-d}a_d(q,u)$$ is a still a polynomial in $u$. By Proposition \ref{propdeg}, we can write
$$\frac{(q^d-1)\cdot\ldots\cdot(q^d-q^{d-1})}{(q-1)(u-1)(u-q)}a_d(q,u)=\sum_{r=0}^{d-2}(u-q^2)\cdot\ldots\cdot(u-q^{d^2-r-1})\left[{d^2\atop r}\right]_q+B(q,u),$$
where $B(q,u)$ has degree at most $d^2-d-1$ in $u$. From identity (\ref{i4}), we easily derive
$$(u-q^2)\cdot\ldots\cdot(u-q^{s-1})=\sum_{k=0}^{s-2}(-1)^kq^{k(k+3)/2}\left[{s-2\atop k}\right]_qu^{s-2-k}.$$ Then, using identity (\ref{i2}) in the second equality and identity (\ref{i1}) in the third equality,
$$\sum_{r=0}^{d-2}(u-q^2)\cdot\ldots\cdot(u-q^{d^2-r-1})\left[{d^2\atop r}\right]_q=$$
$$=\sum_{l=0}^{d^2-2}\sum_{r+k=l}(-1)^kq^{k(k+3)/2}\left[{d^2-r-2\atop k}\right]_q\left[{d^2\atop r}\right]_qu^{d^2-2-l}=$$
$$=\sum_{l=0}^{d^2-2}\sum_{r+k=l}(-1)^kq^{)d^2-r)k}\left[{-d^2+r+1+k\atop k}\right]_q\left[{d^2\atop r}\right]_qu^{d^2-2-l}=$$
$$=\sum_{l=0}^{d^2-2}\left[{l+1\atop l}\right]_qu^{d^2-2-l},$$
and the theorem is proved.\end{proof}

We can also describe the lowest $u$-term in the above expansion.

\begin{proposition} For $d\geq 2$, the constant term of $\overline{a}_d(q,u)$ equals
$$q^{(d+1)(d-2)/2}\frac{1}{d}(q^d-1)\cdot\ldots\cdot(q-1)\sum_{i\cdot j=d}\frac{\mu(i)}{(q^i-1)^j}.$$
In particular, its value at $q=1$ equals $(d-1)!$.
\end{proposition}

\begin{proof} We rewrite identity (\ref{rewrite}) using the variable $z=ut$ as
$$\sum_{d\geq 1}(u^{-d}a_d(q,u))z^d=(1-q)\cdot{\rm Log}(T^{-1}(\sum_{d\geq 0}\frac{u^{d(d-1)/2}z^d}{(1-q)\cdot\ldots\cdot(1-q^d)})^{-1}).$$
Specialization to $u=0$ is well-defined on the right hand side, since
$$T^{-1}(z^d)=(-1)^d(q^{-1}u)^{d(d-1)/2}z^d,$$
and thus
$$T^{-1}(z)=-z\mbox{ and }T^{-1}(z^d)=0\mbox{ for }d\geq 2$$
at $u=0$. Thus we find the identity
$$\sum_{d\geq 1}(u^{-d}a_d(q,u))|_{u=0}z^d=(1-q)\cdot{\rm Log}(1-\frac{z}{q-1}).$$
Using the formula of Section \ref{pre} we compute
$${\rm Log}(1-\frac{z}{1-q})=-\Psi^{-1}\sum_{j\geq 1}\frac{z^j}{j(q-1)^j}=-\sum_{d\geq 1}\sum_{i\cdot j=d}\frac{\mu(i)}{(q^i-1)^j}\frac{z^d}{d}.$$
Thus
$$\overline{a}_d(q,0)=\frac{(q^d-1)\cdot\ldots\cdot(q^d-q^{d-1})}{(q-1)(u-1)(u-q)}u^{-d}a_d(q,u)|_{u=0}=$$
$$=\frac{(q^d-1)\cdot\ldots\cdot(q^d-q^{d-1})}{q(q-1)}(1-q)(-\frac{1}{d}\sum_{i\cdot j=d}\frac{\mu(i)}{(q^i-1)^j}),$$
and we arrive at the claimed formula.

%pecializing $u=0$, we have the defining equation
%$$(1+\frac{z}{q-1})\circ{\rm Exp}(\frac{1}{1-q}\sum_{d\geq 1}\frac{q\overline{a}_d(q,0)}{q^{d(d-1)/2}(q^d-1)\cdot\ldots\cdot(q^2-1)}z^d)=1$$
%where $z^d\circ z^e=0$ whenver $de\geq 1$. Thus we have to compute ${\rm Log}(1-\frac{z}{q-1})$, and this yields the claimed formula.
\end{proof}

\end{document}